\documentclass[11pt]{amsart}

\usepackage {amsmath, amssymb,   epsfig, a4wide, enumerate, psfrag}
\usepackage[utf8]{inputenc}
\usepackage[english]{babel}
\numberwithin{equation}{section}
\date{\today}

\author{Zhuchao Ji}

\title{ Non-wandering Fatou components for strongly attracting polynomial skew products}

\address{Sorbonne Universités, Laboratoire de Probabilités, Statistique et Modélisation (LPSM, UMR 8001),   4 place Jussieu, 75252 Paris Cedex 05, France}
\email{zhuchao.ji@upmc.fr}


\newtheorem{theorem}{Theorem}[section]
\newtheorem{definition}[theorem]{Definition}
\newtheorem{proposition}[theorem]{Proposition}
\newtheorem{corollary}[theorem]{Corollary}
\newtheorem{lemma}[theorem]{Lemma}
\newtheorem{remark}[theorem]{Remark}

\newtheorem*{theorem 4.4}{Theorem 4.3}
\newtheorem*{theorem 6.1}{Theorem 6.1}
\newtheorem*{proposition 5.5}{Proposition 5.5}

\newtheorem*{theorem*}{Theorem}

\begin{document}

	\maketitle
	\begin{abstract}
		We show a partial generalization of Sullivan's non-wandering domain theorem in complex dimension two. More precisely, we show the non-existence of wandering Fatou components for polynomial skew products of $  \mathbb{C}^2$ with an invariant attracting fiber, under the assumption that the multiplier $ \lambda $ is small. We actually show a stronger result, namely that every forward orbit of any vertical Fatou disk intersects a bulging Fatou component.
	\end{abstract}

	\tableofcontents
	\section{Introduction}
	Complex dynamics, also known as Fatou-Julia theory, is naturally subdivided according to these two terms. One is focused on the Julia set. This is the set where chaotic dynamics occurs. The other direction of investigation is concerned with the dynamically stable part - the Fatou set. In this paper we will concentrate on the Fatou theory.
	\par In a general setting, let $ M $ be a complex manifold, and let $ f:M\to M $ be a holomorphic self map. We consider $ f $ as a dynamical system, that is, we study the long-time behavior of the sequence of iterates $ \left\lbrace f^n\right\rbrace_{n\geq 0} $. The \textbf{Fatou set} $ F(f) $ is classically defined as the largest open subset of $ M $ in which the sequence of iterates is normal. Its complement is the \textbf{Julia set} $ J(f) $. A \textbf{Fatou component} is a connected component of $ F(f) $.
	\medskip
	
	In one-dimensional case, we study the dynamics of iterated holomorphic self map on a Riemann surface. The classical case of rational functions on Riemann sphere $ \mathbb{P}^1 $ occupies an important place and produces a fruitful theory. The non-wandering domain theorem due to Sullivan \cite{sullivan1985quasiconformal} asserts that every Fatou component of a rational map is eventually periodic. This result is fundamental in the Fatou theory since it leads to a complete classification of the dynamics in the Fatou set: the orbit of any point in the Fatou set eventually lands in an attracting basin, a parabolic basin, or a rotation domain.
	\medskip
	\par The same question arises in higher dimensions, i.e. to investigate the non-wandering domain theorem for higher dimensional holomorphic endomorphisms on $ \mathbb{P}^k $. A good test class is that of polynomial skew products hence one-dimensional tools can be used.
	\par  A polynomial skew product is a map $ P:\mathbb{C}^2\longrightarrow \mathbb{C}^2 $ of the form
	\begin{equation*}
	P(t,z)=(g(t),f(t,z)),
	\end{equation*}
	where $ g $ is an one variable polynomial and $ f $ is a two variable polynomial. See Jonsson \cite{jonsson1999dynamics} for a systematic study of such polynomial skew products, see also Dujardin \cite{dujardin2016non} for an application of polynomial skew products. 
	
	\medskip
		To investigate the Fatou set of $ P
	$, let $ \pi_1 $ be the projection to the $ t $-coordinate, i.e. \begin{equation*}\pi_1 :\mathbb{C}^2\to \mathbb{C}, \;\;\pi_1(t,z)=t .
	\end{equation*}
	We first notice that $\pi_1(F(P))\subset F(g)$, pass to some iterates of $P$, we may assume that  the points in $F(g)$ will eventually land into an immediate  basin or a Siegel disk (no Herman rings for polynomials), thus we only need to study the following semi-local case:
	\begin{equation*}
	P=(g,f): \Delta\times\mathbb{C}\to \Delta\times\mathbb{C},
	\end{equation*}
	where $g(0)=0$ which means the line $\left\lbrace t=0\right\rbrace $ is invariant and $ \Delta $ is the immediate basin or the Siegel disk of $g$. $P$ is called  an attracting, parabolic or elliptic local polynomial skew product when $g'(0)$ is attracting, parabolic or elliptic respectively.
	
	\medskip
	\par The first positive result is due to Lilov.  Under the assumption that $ 0\leq |g'(0)|<1 $, Koenigs' Theorem and Böttcher's Theorem tell us that the dynamical system is locally conjugated to
	\begin{equation*}
	P(t,z)=(\lambda t,f(t,z)),
	\end{equation*}
	when $ g'(0)=\lambda \neq 0 $, or 
	\begin{equation*}
	P(t,z)=(t^m,f(t,z)),\;m\geq 2,
	\end{equation*}
	when $ g'(0)= 0 $. In the first case the invariant fiber is called attracting and in the second case the invariant fiber is called super-attracting. Now $f$ is no longer a polynomial, and $ f $ can be written as a polynomial in $ z $, 
	\begin{equation*}
	f(t,z)=a_0(t)+a_1(t)z+\cdots+a_d(t)z^d,
	\end{equation*}
	with coefficients $ a_i(t) $ holomorphic in $ t $ in a neighborhood of $ 0 $, we further assume that $ a_d(0)\neq 0 $ (and we make this assumption in the rest of the paper). In this case the dynamics in $\left\lbrace  t=0\right\rbrace  $ is given by the polynomial
	\begin{equation*}
	p(z)=f(0,z)
	\end{equation*}
	and is very well understood. In his unpublished PhD thesis \cite{lilov2004fatou}, Lilov first showed that  every Fatou component of $ p $ in the super-attracting invariant fiber is contained in a two-dimensional Fatou component, which is called a bulging  Fatou component. We will show that this bulging property of Fatou component of $ p $ also holds in attracting case.  
	\par Lilov's main result is the non-existence of wandering Fatou components for local polynomial skew products in the basin of a super-attracting invariant fiber. Since this is a local result, it can be stated as follows.
	\begin{theorem*}[\bf Lilov] \label{Lilov}
		For a local polynomial skew product $ P $ with a super-attracting invariant fiber,
		\begin{equation*}
		P(t,z)=(t^m,f(t,z)),\;m\geq 2,
		\end{equation*}
 every forward orbit of a vertical Fatou disk intersects a bulging Fatou component. This implies that every Fatou component iterate to a bulging Fatou component. In particular there are no wandering Fatou components.
	\end{theorem*}
See Definition 2.1 for the definition of the vertical Fatou disk.
	\medskip
	\par On the other hand, recently Astorg, Buff, Dujardin, Peters and Raissy \cite{astorg2014two} constructed a holomorphic endomorphism $ h:\mathbb{P}^2\longrightarrow \mathbb{P}^2$, induced by a polynomial skew product $ P=(g(t),f(t,z)):\mathbb{C}^2\longrightarrow \mathbb{C}^2 $ with parabolic invariant fiber, processing a wandering Fatou component, thus the non-wandering domain theorem does not hold for general polynomial skew products.
	\medskip
		\par At this stage it remains an interesting problem to investigate the existence of wandering Fatou components for local polynomial skew products with attracting but not super-attracting invariant fiber. As it is clear from Lilov's theorem, Lilov actually showed a stronger result, namely that every forward orbit of a vertical Fatou disk intersects a bulging Fatou component. Peters and Vivas showed in \cite{peters2016polynomial} that there is an attracting local polynomial skew product with a wandering vertical Fatou disk, which shows that Lilov's proof breaks down in the general attracting case. Note that this result does not answer the existence question of wandering Fatou components, but shows that the question is considerably more complicated than in the super-attracting case. On the other hand, by using a different strategy from Lilov's, Peters and Smit in \cite{peters2017fatou} showed that the non-wandering domain theorem holds in the attracting case, under the assumption that the dynamics on the invariant fiber is sub-hyperbolic. The elliptic case was studied by Peters and Raissy in \cite{peters2017elliptic}. See also Raissy \cite{raissy2017polynomial} for a survey of the history of the investigation of wandering domains for polynomial skew products.
	\medskip
	\par In this paper we prove a non-wandering domain theorem in the attracting local polynomial skew product case  without any assumption of the dynamics on the invariant fiber. Actually we show that Lilov's stronger result holds in the attracting case when the multiplier $ \lambda $ is small.
	\begin{theorem*}[\bf Main Theorem] \label{Main Theorem}
		For a local polynomial skew product $ P $ with an attracting invariant fiber,
		\begin{equation*}
		P(t,z)=(\lambda t,f(t,z)),
		\end{equation*}
		for any fixed $ f $, there is a constant $ \lambda_0 =\lambda_0(f)>0$ such that if $ \lambda $ satisfies $ 0<|\lambda|<\lambda_0 $, every Fatou component iterates to a bulging Fatou component. In particular there are no wandering Fatou components.
	\end{theorem*}
\par We can also apply this local result to globally defined polynomial skew products,  see Theorem 6.4 for the precise statement.
	\par The proof of the main theorem basically follows Lilov's strategy. The difficulty is that Lilov's argument highly depends on the super-attracting condition and breaks down in the attracting case by \cite{peters2016polynomial}. The main idea of this paper are to use and adapt an one-dimensional lemma due to Denker-Przytycki-Urbanski(the DPU Lemma for short) to our case. This will give estimates of the horizontal size of bulging Fatou components and of the size of forward images of a wandering vertical Fatou disk (these concepts will be explained later). We note that some results in our paper already appear in Lilov's thesis \cite{lilov2004fatou} (Theorem 3.4, Lemma 4.1, Lemma 5.1 and Lemma 5.2). Since his paper is not easily available, we choose to present the whole proof with all details. On the other hand we believe that the introduction of the DPU Lemma makes the argument conceptually simpler even in the super-attracting case.
	\medskip
	\par The outline of the paper is as follows. In section 2 we start with some definitions, then we present the DPU Lemma and some corollaries. In section 3 we show that every Fatou component of $ p $ in the invariant fiber bulges, i.e. is contained in a two-dimensional bulging Fatou component. This result follows classical ideas from normal hyperbolicity theory.
	\par In section 4 we give an estimate of the horizontal size of the bulging Fatou components by applying the one-dimensional DPU Lemma. Let $ z\in F(p) $ be a point in a Fatou component of the invariant fiber and denote by $ r(z) $ the supremum  radius of a horizontal holomorphic disk (see Definition 2.1 for the precise definition) centered at $z$ that is contained in the bulging Fatou component. We have the following key estimate.
	\begin{theorem 4.4}
		If $\lambda$ is chosen sufficiently small, then there are constants $ k>0 $, $ l>0 , R>0$ such that for any point $ z\in F(p)\cap\left\lbrace |z|<R\right\rbrace  $, 
		\begin{equation*}
		r(z)\geq k\; d(z,J(p))^l,
		\end{equation*}
		where $ J(p) $ is the Julia set in the invariant fiber.
	\end{theorem 4.4}
	\medskip
	\par In section 5 we adapt the DPU Lemma to the attracting local polynomial skew product case, to show that the size of forward images of a wandering vertical Fatou disk shrinks slowly, which is also important in the proof of the main theorem. 
	\begin{proposition 5.5}
		Let $ \Delta_0 \subset \left\lbrace t=t_0\right\rbrace  $ be a wandering vertical Fatou disk centered at $ x_0=(t_0,z_0) $. Let $x_n= (t_n,z_n)=P^n(x_0) $. Define a function $ \rho $ as follows: for a domain $ U \subset \mathbb{C} $, for every $ z\in U \subset \mathbb{C}$, define
		\begin{equation*}
		\rho(z,U)=\sup\left\lbrace r>0 |\;D(z,r)\subset U\right\rbrace .
		\end{equation*}
		Set $ \Delta_n=P^n(\Delta_0) $  for every $ n \geq 1$ and let $  \rho_n=\rho(z_n,\pi_2(\Delta_n)) $, here $ \pi_2 $ is the projection $ \pi_2:(t,z)\mapsto z $. If $\lambda$ is chosen sufficiently small, we have 
		\begin{equation*}
		\lim_{n\to \infty}\frac{|\lambda|^n}{\rho_n}=0.
		\end{equation*}
		
	\end{proposition 5.5}
	\medskip
	\par The proof of the main theorem is given in section 6. The main point are to combine Theorem 4.3 and Proposition 5.5 to show wandering vertical Fatou disk can not exist. We finish section 6 with some remarks around the main theorem. We also show how our main theorem can be applied to globally defined polynomial skew products in theorem 6.4.
	\medskip
	\par \textbf{Acknowledgements.} I would like to thank Romain Dujardin for drawing my attention to the subject and for his invaluable help. The work is partially supported by ANR-LAMBDA, ANR-13-BS01-0002. I also would like to thank the referee for the nice suggestions on the structure of the paper.
	
	\section{Preliminaries}
	\subsection{Horizontal holomorphic disk and vertical Fatou disk}
	In this subsection we make the precise definitions appearing in the statement of Theorem 4.3 and Proposition 5.5.
	\medskip
	\par Recall that after a local coordinate change our map has the form $ P: \Delta\times\mathbb{C}\to \Delta\times\mathbb{C} $, here $ \Delta\subset\mathbb{C}$ is a disk centered at $ 0 $, such that
	\begin{equation*}
	P(t,z)=(\lambda t, f(t,z)),
	\end{equation*}
	here $f$ is a polynomial in $ z $ with coefficients $a_i(t)$ holomorphic in $ \Delta $, and $ a_d(0)\neq 0 $, $ \lambda $ satisfies $ 0<|\lambda|<1 $. 

\bigskip
	\begin{definition}
\begin{itemize}	 
		\item 	A \textbf{horizontal holomorphic disk} is a subset of the form
		\begin{equation*}
		\left\lbrace (t,z)\in \Delta\times\mathbb{C}\;| \;z=\phi(t),\;|t|<\delta\right\rbrace 
		\end{equation*}
		where $ \phi(t) $ is holomorphic in $ \left\lbrace |t|<\delta\right\rbrace $ for some $ \delta>0 $. $ \delta $ is called the \textbf{size} of the horizontal holomorphic disk.
		\medskip
		\item  Let $ \pi_2 $ denote the projection to the $ z $-axis, that is
		\begin{align*}
		\pi_2: \Delta\times\mathbb{C}&\longrightarrow\mathbb{C},\\
		(t,z)&\longmapsto z.
		\end{align*}
		A subset $ \Delta_0 $ lying in some $ \left\lbrace t=t_0\right\rbrace $ is called a \textbf{vertical disk} if $ \pi_2\left( \Delta_0\right) $  is a disk in the complex plane. A vertical disk centered at $ x_0 $ with radius $ r $ is denoted by $ \Delta(x_0,r) $. $ \Delta_0 $ is called a  \textbf{vertical Fatou disk } if the restriction of $ \left\lbrace P^n\right\rbrace_{n\geq0}$ to $\Delta_0 $ is a normal family.
	
		\end{itemize}
	\end{definition}
\medskip
In the rest of the paper, for a disk on the complex plane centered at $ z $ with radius $ r $, we denote it by $ D(z,r) $ to distinguish.  
\medskip
	\begin{remark}
		A vertical disk contained in a Fatou component of $ P $ is a vertical Fatou disk.
	\end{remark}
\medskip
\par We define a positive real valued function $ r(z) $, which measures the horizontal size of the bulging Fatou components.
\begin{definition}
	For $ z\in\mathbb{C} $ satisfying $ |z|<R $ and $ z $ lying in the Fatou set of $ p $, we define $ {r(z)} $ to be the supremum of all positive real numbers $ r $ such that there exist a horizontal holomorphic disk passing through $ z $ with size $ 2r $, contained in $ F(P)\cap \left\lbrace |z|\leq R\right\rbrace $. 
	
\end{definition}

\medskip
	\subsection{Denker-Przytycki-Urbanski's Lemma}
	In this subsection we introduce the work of Denker-Przytycki-Urbanski in \cite{denker1996transfer}, and give some corollaries. Denker, Przytycki and Urbanski
	consider rational maps on $ \mathbb{P}^1 $, and study the local dynamical behavior of some neighborhood of a critical point lying in Julia set. As a consequence they deduce an upper bound of the size of the pre-images of a ball centered at a point in Julia set.
	\medskip
	\par  In the following  let $ f $ be a rational map on $ \mathbb{P}^1 $, denote by $ C(f) $ the set of critical points of $ f $ lying in Julia set.  Assume that $\# C(f)=q $. We begin with a definition,
	
	\begin{definition}
		For a critical point $ c\in C(f) $, define a positive valued function $ {k_c(x)} $ by 
		\begin{equation*}
		k_c(x)=\left\{
		\begin{aligned}
		-&\log  d(x,c), & \; \text{if}\;x\neq c  \\
		&\infty, &\;\text{if}\; x=c. \\
		\end{aligned}
		\right.
		\end{equation*}
		Define a function $ {k(x)} $ by
		\begin{equation*}
		k(x)=\max_{c\in C(f)} k_c(x).
		\end{equation*}
		Here the distance is relative to the spherical metric on $ \mathbb{P}^1 $. 
	\end{definition}
\medskip
	Let $ x_0 $ be arbitrary and consider the forward orbit $ \left\lbrace x_0, x_1,\cdots, x_n,\cdots\right\rbrace $, where $ x_n=f^n(x_0) $. We let the function $ k(x) $ acts on this orbit and the following DPU Lemma gives an asymptotic description of the sum of $ k(x) $ on this orbit. Recall that $ q $ denotes the number of critical points lying in $J$.
	\begin{lemma}[\bf Denker, Przytycki, Urbanski] \label{Denker, Przytycki, Urbanski}
		There exist a constant $ Q>0 $ such that for every $ x\in\mathbb{P}^1$, and $ n\geq 0 $, there exists a subset $ \left\lbrace j_1, \cdots, j_{q'}\right\rbrace \subset \left\lbrace 0, 1, \cdots, n \right\rbrace $, such that
		\begin{equation*}
		\sum_{j=0}^{n}k(x_j)-\sum_{\alpha=1}^q k(x_{j_\alpha})\leq Qn,
		\end{equation*}
		here $ q'\leq q $ is an integer.
	\end{lemma}
\medskip
	Lemma 2.5 implies that in a sense the orbit of a point can not come close to $ C(f) $ very frequently. As a consequence Denker, Przytycki, Urbanski deduce an upper bound of the size of the pre-images of a ball centered at a point in $ J(f) $.
	\begin{corollary}[\bf Denker, Przytycki, Urbanski] \label{Denker, Przytycki, Urbanski}
		There exist $ s\geq 1 $ and $ \rho >0 $ such that for every $ x\in J(f) $, for every $ \epsilon>0$, $ n\geq 0 $, and for every connected component $ V $ of $ f^{-n}(B(x,\epsilon)) $, one has $\text{diam}\;V \leq s^n\epsilon^\rho $.
		
	\end{corollary}
	
	\begin{corollary}
		Let $ f $ be a polynomial map in $ \mathbb{C} $. For fixed $ R>0 $, there exist $ s\geq 1 $ and $ \rho >0 $ such that for any $ n\geq 0 $ and any $ z\in \mathbb{C} $ satisfying $ f^n(z)\in \left\lbrace |z|<R \right\rbrace $, we have 
		\begin{equation*}
		d(z,J(f))\leq s^nd(f^n(z),J(f))^\rho,
		\end{equation*}
		where the diameter is relative to the Euclidean metric.
	
	\end{corollary}
	\begin{proof}
Since the Euclidean metric and the spherical metric are equivalent on a compact subset of $ \mathbb{C} $, by Corollary 2.6 for fixed $ R>0 $, there exist $ s\geq 1 $ and $ \rho >0 $ such that for every $z$ satisfying  $z\in J(f)$, $ 0<\epsilon\leq R $, $ n\geq 0 $, and for every connected component $ V $ of $ f^{-n}(D(z,\epsilon)) $, one has $\text{diam}\;V \leq s^n\epsilon^\rho $.
	
		For any $ z $ and $ n $ satisfy $ f^n(z)\in \left\lbrace |z|<R \right\rbrace $, let $ y\in J $ satisfy $ d(f^n(z),J(f))=d(y,f^n(z))=\epsilon $.  For every connected component $ V $ of $ f^{-n}(D(y,2\epsilon)) $, one has $\text{diam}\;V \leq s_1^n(2\epsilon)^\rho $, so that 
		\begin{align*}
		d(z,J(f))\leq d(z,f^{-n}(y))\leq \text{diam}\;V \leq s_1^n2^\rho\epsilon^\rho .
		\end{align*}
		Set $ s= 2^\rho s_1 $ and the proof is complete.
	\end{proof}
\begin{remark}
 The existence of such a result is intuitive since the Julia set is expected to be repelling in some sense - however the presence of critical points on $J$ makes it non-trivial.
\end{remark}

	\section{Structure of bulging Fatou components}
	In this section we show that every Fatou component of $ p $ in the invariant fiber is actually contained in a Fatou component of $ P $, which is called a \textbf{bulging Fatou component}, and in this case we call the Fatou component of $ p $  \textbf{bulges}. By Sullivan's theorem every Fatou component of $ p $ is pre-periodic, it is sufficient to show that every periodic Fatou component  of $ p $ is contained in a Fatou component of $ P $. There are three kinds of periodic Fatou components of $ p $, i.e. attracting basin, parabolic basin and Siegel disk. For all these three kinds we study the structure of the associated bulging Fatou components.
	
	\par We may iterate $ P $ many times to ensure that all periodic Fatou components of $ p $ are actually fixed, and all parabolic fixed points have multiplier equals to 1. In the following of this paper the metric referred to is the Euclidean metric.
	\medskip
	\subsection{Attracting basin case} In the attracting basin case, assume that we have an attracting basin $ B $ of $ p $ in the invariant fiber. Without loss generality we may assume $ 0 $ is the fixed point in $B$, so that $ (0,0) $ becomes a fixed point of $ P $, and $ p'(0)=\lambda' $ with $ |\lambda'|<1 $.  We have the following well-known theorem \cite{rosay1988holomorphic}.
	\begin{theorem}
		If $ P: \Omega\to \Omega $ is a holomorphic self map, where $ \Omega $ is an open set of $ \mathbb{C}^2 $ and $ (0,0)\in\Omega $ is a fixed point. If all eigenvalues of the derivative $ DP(0,0) $ are less than $ 1 $ in absolute value then $ P $ has an open attracting basin at the origin. 
	\end{theorem}
	In our case we have $$ DP(0,0)=\begin{pmatrix}
	\lambda	&  0\\ 
	\frac{\partial f}{\partial t}(0,0)	& \lambda'
	\end{pmatrix}, $$
	so that all the all eigenvalues of the derivative $ DP(0,0) $ are less than $ 1 $ in absolute value. As a consequence $ B$ is contained in a two dimensional attracting basin of $ (0,0) $, say $ U$, so that $ B $ bulges.
	
	\medskip
	\subsection{Parabolic basin case} In the parabolic basin case suppose $0$ is a parabolic fixed point of $ p $. Assume that $ p $ is locally conjugated to $ z\mapsto z+az^s+O(z^{s+1}) $ for some $ s\geq 2 $, $ a\neq 0 $. We first prove that near the fixed point $(0,0)$, $P $ is locally conjugated to 
	\begin{equation*}
	(t,z)\mapsto(\lambda t,z+az^s+O(z^{s+1})).
	\end{equation*}
	where $ O(z^{s+1}) $ means there are constant $ C $ such that the error term $ \leq C|z|^{s+1} $, for all $ (t,z) $ in a neighborhood of the origin. Then we prove in this coordinate every parabolic basin of $ p $ bulges.
	\begin{lemma}
		Assume $ (0,0) $ is a fixed point of $ P $, and $ |p'(0)|=1 $, then there exist a stable manifold through the origin in the horizontal direction. More precisely, there is a holomorphic function $ z=\phi(t) $ defined on a small disk $ \left\lbrace |t|<\delta\right\rbrace $ such that
		\begin{equation*}
		\phi(0)=0,\;\text{and}\;\;f(t,\phi(t))=\phi(\lambda t).
		\end{equation*}
	\end{lemma}
	\begin{proof}
	This is related to the two dimensional Poincaré's theorem. See \cite{ueda1986local} Theorem 3.1 for the proof.
	\end{proof}
\medskip
We have the following theorem which is a special case of \cite[\S 7.2]{ueda1986local}.
	\begin{theorem}
		We assume that the local skew product is given by
		\begin{equation*}
		P(t,z)=(\lambda t,z+az^s+O(z^{s+1})),
		\end{equation*}
		then there exist a constant $ \delta>0 $ and s-1 pairwise disjoint simply connected open sets $ U_j\subset \left\lbrace |t|<\delta\right\rbrace \times \mathbb{C} $, referred to as two dimensional attracting petals, with the following properties:\\
		\par (1) $ P(U_j)\subset U_j, $ points in $ U_j $ converge to $ (0,0) $ locally uniformly.\\
		\par (2) For any point $ x_0=(t_0,z_0) $ such that $ P^n(x_0)\rightarrow (0,0) $, there exist integer $ N $ and $ j $ such that for all $ n\geq N $ either $ P^n(x_0)\in U_j $ or $ z_n=0 $.\\
		\par (3) $  U_j=\left\lbrace |t|<\delta\right\rbrace \times\left( U_j\cap \left\lbrace t=0\right\rbrace \right)  $.
	\end{theorem}

\medskip
	\par Thus by Theorem 3.3, for fixed $ j $, all the points $ x_0$ whose orbit finally lands on $ U_j $ form an open subset $ \Omega_j $,  which is contained in the Fatou set of $ P $. It is obvious that every parabolic basin of $ p $ is contained in one of such $ \Omega_j $, this implies all parabolic basins of $ p $ bulge.
	\medskip
	\subsection{Siegel disk case} In the Siegel disk case, we assume that $ 0 $ is a Siegel point with a Siegel disk $D\subset \left\lbrace t=0\right\rbrace$. We are going to prove that $ D $ is contained in a two dimensional Fatou component.
	\begin{theorem}
		Assume that p is locally conjugated to $ z\mapsto e^{i\theta}z $ with $ \theta $ an irrational multiple of $ \pi_2 $, then there is a neighborhood $ \Omega $ of $ D $ such that $ D\subset\Omega\subset\mathbb{C}^2 $, and there exists a biholomorphic map $ \psi $ defined on $ \Omega $ such that
		\begin{equation*}
		\psi\circ P\circ\psi^{-1}(t,z)=(\lambda t,e^{i\theta}z).
		\end{equation*}
		
	\end{theorem}
	\begin{proof}
		We may assume that $ p $ is conjugated to $ z\mapsto e^{i\theta}z $, then by Lemma 3.2 there is a stable manifold $ z=\phi(t) $. A change of variables $ z\mapsto z+\phi(t) $ straightens the stable manifold so that $ P $ is conjugated to
		\begin{equation*}
		(t,z)\to (\lambda t,e^{i\theta}z+tg(t,z)),
		\end{equation*} 	where $ g(t,z) $ is a holomorphic function. By an abuse of notation we rename this map by $ P $. 
	
		\par Let $U $ be a relatively compact neighborhood of $ \overline{D} $ in $ \mathbb{C}^2 $. Set $ C=\sup|g(t,z)| $ on $ U $. Let $ \delta$ be so small that $ \frac{C\delta}{1-\delta}<\text{dist}(D,\partial U)  $, and then $\Omega=\left\lbrace |t|<\delta\right\rbrace \times D $ is an open subset of $ U $. Let $ (t_0,z_0) $ be an arbitrary initial point in $ \Omega $, and denote $ P^n(t_0,z_0) $ by $ (t_n,z_n) $, then
		\begin{equation*}
		\left|  |z_{n+1}|-|z_{n}|\right|  \leq |t_ng(t_n,z_n)|\leq C|\lambda|^n\delta.
		\end{equation*}
Then we have
		\begin{equation*}
		\left|  |z_{n}|-|z_0|\right|  \leq \frac{C\delta}{1-\delta} +|z_0| \leq \text{dist}(z_0,\partial U)+|z_0|,
		\end{equation*}
		so that $ (t_n,z_n) $ still lies in $ U $. Thus $\left\lbrace  P^n\right\rbrace  $ is a normal family on $ \Omega $, for the reason that $ P^n(\Omega) $ is uniformly bounded.
		\par Thus we can select a sub-sequence $ \left\lbrace n_j\right\rbrace $ for which the sequence
		\begin{equation*}
		\phi_{t_0}(z_0)=\lim_{j\rightarrow \infty} e^{-in_j\theta}f_{t_{n_j}}\circ f_{t_{n_j}-1}\circ\cdots\circ f_{t_0}(z_0) 
		\end{equation*}uniformly converges on compact subset of $ \Omega $. Thus $ \phi_t(z) $ is a holomorphic function on $ \Omega $, and we have
		\begin{align*}
		\phi_{\lambda t_0}\circ f_{t_0}(z_0)=e^{i\theta}\phi_{t_0}(z_0)
		\end{align*}
		for every $ (t_0,z_0)\in\Omega $. Thus if we let $ \psi(t,z)=(t,\phi_t(z)) $, since $ \phi_0(z)=z $ we can shrink $ \Omega $ if necessary to make sure that $ \psi $ is invertible on $ \Omega $, and we have
		\begin{equation*}
		\psi\circ P\circ\psi^{-1}(t,z)=(\lambda t,e^{i\theta}z).
		\end{equation*}
		For every $ (t,z)\in\Omega $.
	\end{proof} 
\medskip
	\par  It is obvious that $ \Omega $ is contained in the Fatou set of $ P $. Since $ D\subset\Omega $, this implies that every Siegel disk of $ p $ bulges. 
	\medskip
	\subsection{Wandering vertical Fatou disks}
	
	We finish section 3 with a definition.
	\begin{definition}
	A vertical Fatou disk $ \Delta  $ is called \textbf{wandering} if the forward images of $ \Delta $ do not intersect any bulging Fatou component.
	\end{definition}
\medskip
We note that  "wandering" has special meaning in our definition. The definition of wandering vertical Fatou disk we made here is not equivalent to vertical Fatou disks containing  wandering points.
\begin{remark}
The forward orbit of a wandering vertical Fatou disk clusters only on $ J(p) $.
\end{remark}
\begin{proof}
	This is simply because for every $ x=(t,z)\in \Delta$, if $ P^n(x) $ tends to $ (0,z_0)\in F(p)$ then eventually $ P^n((t,z)) $ lands in the bulging Fatou component that contains $ (0,z_0) $. This contradicts the fact  $ x $ lying in a wandering Fatou disk.
\end{proof}
	\section{Estimate of horizontal size of bulging Fatou components}
	In this section we deduce an estimate of the horizontal size of the bulging Fatou components, by applying the one-dimensional DPU Lemma. 
	\par  In the following we choose $ R>0 $ such that if $ (t,z) $ satisfies $ t \in\Delta$, $ |z|>R$, then $ |f(t,z)|\geq 2|z| $. This follows that the line at infinity is super-attracting. Thus for any holomorphic function $ \phi(t) $ defined on $ \left\lbrace |t|<r\right\rbrace $ such that $|\phi(t)|\leq R$, we have for all $ |t|<r $, 
	\begin{equation}
	|\phi(t)-\phi(0)|\leq 2R \frac{|t|}{r},
	\end{equation}
		this follows from the classical Schwarz Lemma.
			\medskip
	\par We begin with a lemma.

	\begin{lemma}
		Let $  \text{Crit}(P)=\left\lbrace (t,z)|\;\frac{\partial f}{\partial z}(t,z)=0\right\rbrace $, then there exist constants  $ 0<\delta_1<1$ and $ K>K_1>0 $ such that any connected component $ C_k $ of $\text{Crit}(P)\cap \left\lbrace |t|<\delta_1\right\rbrace  $  intersects the line $ \left\lbrace t=0\right\rbrace $  in a unique point, say $ c_k $, and for any point $ x=(t,z)\in \text{Crit}(P) $, say $x\in C_k $, we have
		\begin{equation}
		|z-c_k|\leq K_1 |t|^{\frac{1}{d_1}}.
		\end{equation}
		and
		\begin{equation}
		|f(t,z)-p(c_k)| \leq K|t|^{\frac{1}{d_1}},
		\end{equation}
		where $ c_k= C_k\cap\left\lbrace t=0\right\rbrace $, and $ d_1 $ is the maximal multiplicity of critical points of $ p $. 
	\end{lemma}
	\begin{proof}
		Since $ \text{Crit}(P) $ is an analytic variety,  by Weirstrass preparation theorem we can let $ \delta_1 <1$ small enough so that $\text{Crit}(P)\cap \left\lbrace |t|<\delta_1\right\rbrace =\cup_{k=1}^l C_k $ where $ C_k $, $1\leq k\leq l$ are local connected analytic sets, $ C_k\cap \left\lbrace t=0\right\rbrace =\left\lbrace c_k\right\rbrace $. For each fixed component $ C$ intersect $ \left\lbrace t=0\right\rbrace $ at $c$, $ C $ is given by the zero set of a Weirstrass polynomial,
		\begin{equation*}
		C=\left\lbrace (t,z)\in \left\lbrace |t|<\delta_1\right\rbrace \times\mathbb{C}, g(t,z)=0\right\rbrace,
		\end{equation*}
		where $ g(t,z)=(z-c)^{m}+a_{m-1}(t) (z-c)^{m-1}+\cdots+a_0(t) $ is a Weirstrass polynomial, $ m\leq d_1 $ is an integer, $ a_i(t) $ are holomorphic functions in $ t $ satisfying  $ |a_i(t)|\leq M|t| $ for some constant $ M>0 $ .
		\par We show that 
		\begin{equation*}
		|z-c|\leq mM |t|^{\frac{1}{m}}.
		\end{equation*}
		We argue by contradiction. Suppose there exist a point $ (t_0,z_0)\in C $ such that $ \frac{|z_0-c|}{|t_0|^{\frac{1}{m}}}=a>mM $,  then we have
	\begin{equation*}
	|z_0-c|^m=a^m|t_0|,
	\end{equation*}
	and
	\begin{align}
	|a_{m-1}(t_0) (z_0-c)^{m-1}+\cdots+a_0(t_0) |\leq mMa^{m-1}|t_0|.
	\end{align}
	Thus we have $ |z_0-c|^m> 	|a_{m-1}(t_0) (z_0-c)^{m-1}+\cdots+a_0(t_0) | $, which contradicts to $ (t_0,z_0)\in C $ .
	Setting $ K_1=2d_1M$ we get (4.2).
		\par Let $ \Omega $ be a relatively compact open set that contains $ \text{Crit}(P)\cap \left\lbrace |t|<\delta_1\right\rbrace$. Let 
		\begin{equation*}
		M'=\max\left\lbrace \left| \frac{\partial f}{\partial z}\right| ,\left| \frac{\partial f}{\partial t}\right|:\;(t,z)\in \Omega\right\rbrace .
		\end{equation*}
		Then for $ (t,z)\in C_k $ we have
		\begin{align*}
		|f(t,z)-p(c_k)|&\leq M'|t|+M'|z-c_k|\\&\leq M'(1+K_1)|t|^{\frac{1}{d_1}}.
		\end{align*}
		To get (4.3) we set $ K=2\max\left\lbrace M'(1+K_1),2R\right\rbrace  $. Thus the proof is complete.
	\end{proof}
\begin{remark}
We note that $ K_1$ and $K$ are invariant under a local coordinate change of the form $ t\mapsto \phi(t) $ with $ \phi(0)=0 $ and $ \phi'(0)=1 $. To see this let $ a_i(t) $ be the coefficients of the Weirstrass polynomial, the coordinate change $ t\mapsto \phi(t) $ with $ \phi'(0)=1 $ takes $ a_i(t) $ become $ a_i(\phi(t)) $. We have $ |a_i(\phi(t)|\leq 2M|t| $  by shrinking $ \delta_1(\phi) $ if necessary , then we get (4.2) with the same constant $ K_1 $ (this is the reason for the constant 2 in definition of $ K_1 $). By shrinking $ \delta_1(\phi) $ we see that $ \Omega $ and $ R $ are invariant, and 
\begin{equation*}
\max\left\lbrace \left| \frac{\partial f(\phi(t),z)}{\partial z}\right| ,\left| \frac{\partial f(\phi(t),z)}{\partial t}\right|:\;(t,z)\in \Omega\right\rbrace\leq 2M' .
\end{equation*}
 By the same reason we get (4.3) with the same constant $ K $.
\end{remark}

	\medskip
	We are going to prove the following estimate of $ r(z) $ under the assumption that the multiplier $ \lambda $ is sufficiently small.
	\begin{theorem}
		There exist a constant $ \lambda_1=\lambda_1(f)>0 $ such that for fixed $ |\lambda|<\lambda_1$, there are constants $ k>0 $, $ l>0 $ such that for any point $ z\in F(p)\cap\left\lbrace |z|<R\right\rbrace  $, 
		\begin{equation*}
		r(z)\geq k\;d(z,J(p))^l,
		\end{equation*}
		here $ J(p) $ is the Julia set of $ p $ in the invariant fiber. Furthermore $l$ depends only on $p$.
	\end{theorem}
	We would like to give an outline of the proof of Theorem 4.3 first. Since there are only finitely many invariant Fatou components of $ p $, and every Fatou component is pre-periodic to one of them, it is enough to prove Theorem 4.3 holds for $ z $ in the basin of an invariant Fatou component. To do this, we first fix an invariant Fatou component $ U $, and we prove Theorem 4.3 holds for a subset  $W$ satisfying $ \cup_{i=0}^\infty p^{-i}(W) $= the basin of $U$, this is the first step. In step 2,  we use the following Pull Back Lemma to get the relation between $ r(z) $ and $ r(p(z)) $, together with the DPU Lemma we are able to give the estimate for the points in $ p^{-i}(W) $, for every $i$. We start with the Pull Back Lemma.
	
	\begin{lemma}[\bf Pull Back lemma] \label{Pull Back Lemma}
		There exist a constant $ 0<\epsilon<1 $, such that if we let 
		\begin{equation*}
		V=\left\lbrace z\in F(p),d(z,J(p))<\epsilon \right\rbrace ,
		\end{equation*}
		then for any $ z_0\in F(p)\cap \left\lbrace |z|<R\right\rbrace $ such that $ p(z_0)\in V $, at least one of the following holds:
		\begin{equation}
		r(z_0)\geq \frac{\alpha}{|\lambda|} r(p(z_0))d(z_0,C(p))^{d_1(d_1+1)};
		\end{equation}
		or
		\begin{equation}
		r(z_0)\geq \beta \;d(z_0,J(p))^{d_1(d_1+1)}.
		\end{equation}
		Here $\alpha, \beta $ are positive constants only depending on $ p $ and the constant $ K $ from Lemma 4.1, and $ C(p) $ is the set of critical points lying in $J(p)$.
	\end{lemma}
	\begin{proof}
		Let $ \text{Crit}(p) $ be the set of critical points of $ p $, We choose $ \epsilon $ small such that $ p(z)\in V $ implies 
		$ d(z,p(C(p)))= d(z,p(\text{Crit}(p)) $. Let $ \phi $ be the associated holomorphic function with respect to $ p(z_0) $ with size $ r(p(z_0)) $. We are going to show that the critical value set of $P$ does not intersect the graph of $ \phi $ when the domain of $ \phi $ is small .
		\par Suppose  $x'= (t',z') $ lies in $ \text{Crit}(P) $ satisfying $ t'<r(z_0) $ and $ P(x') $ lying in the graph of $ \phi $. then by Lemma 4.1 the connected component of $ \text{Crit}(P) $ containing $ x' $ intersects $ \left\lbrace t=0 \right\rbrace  $ at a unique point $ c $. Then we have
		\begin{align}
		d(p(z_0),p(C(p))&\leq |p(z_0)-p(c)|\notag\\&=|\phi(0)-\phi(\lambda t')+f(t',z')-p(c)|\notag\\&\leq |\phi(0)-\phi(\lambda t')|+|f(t',z')-p(c)|\notag\\&\leq K \frac{|\lambda t'|}{r(p(z_0))}+K|t'|^{1/d_1}.
		\end{align}
		(4.7) holds by applying Lemma 4.1 and inequality (4.1).
		\par Now there are two possibilities,
		
		\par (a) If $ \frac{|\lambda t'|}{r(p(z_0))}\geq 1 $, then
		\begin{equation*}
		|t'|\geq \frac{r(p(z_0))}{|\lambda|}.
		\end{equation*}
		\par (b) If $ \frac{|\lambda t'|}{r(p(z_0))}<1 $, then
		\begin{equation*}
		\frac{|\lambda t'|}{r(p(z_0))}\leq \frac{|\lambda t'|^{1/d_1}}{r(p(z_0))^{1/d_1}},
		\end{equation*}
		so that by (4.7) we have
		\begin{equation*}
		d(p(z_0),p(C(p))\leq K \frac{|\lambda t'|^{1/d_1}}{r(p(z_0))^{1/d_1}}+K|t'|^{1/d_1}.
		\end{equation*}
		\par  For case (b), there are two subcases,
		\par (b1) If $r(p(z_0))\leq |\lambda| $, then
		\begin{equation*}
		d(p(z_0),p(C(p))\leq 2K \frac{|\lambda t'|^{1/d_1}}{r(p(z_0))^{1/d_1}},
		\end{equation*}
		by applying the fact that there is a constant $ c=c(p)>0 $ such that $ d(p(z_0),p(C(p))\geq c\; d(z_0,C(p))^{d_1+1} $ we have
		\begin{equation*}
		|t'|\geq \frac{\alpha}{|\lambda|} r(p(z_0))d(z_0,C(p))^{d_1(d_1+1)},
		\end{equation*}
		where $ \alpha= \left( \frac{c}{2K}\right) ^{d_1} $.
		
		\par (b2) If $ r(p(z_0))>|\lambda| $, then
		\begin{equation*}
		d(p(z_0),p(C(p))< 2K | t'|^{1/d_1}.
		\end{equation*}
		Thus we have 
		\begin{equation*}
		|t'|>\left( \frac{1}{2K} \right) ^{d_1} d(p(z_0),J(p))^{d_1}.
		\end{equation*}
By applying the fact that there is a constant $ c=c(p)>0 $ such that $ d(p(z_0),J(p))\geq c d(z_0,J(p)^{d_1+1} $, we get
		\begin{equation*}
		|t'|\geq \beta \;d(z_0,J(p))^{d_1(d_1+1)}.
		\end{equation*}
		where $ \beta=c\left( \frac{c}{2K}\right) ^{d_1}$.
		\par We can let $ \alpha $ small enough such that actually $ \alpha d(z_0,C(p))^{d_1(d_1+1)}<1 $, thus for case (b1) we have
		\begin{equation*}
		|t'|\geq \frac{\alpha}{|\lambda|} r(p(z_0))d(z_0,C(p))^{d_1(d_1+1)}\geq \frac{r(p(z_0))}{|\lambda|},
		\end{equation*}
		thus case (a) is actually contained in case (b1).
		\par In either case (b1) or (b2) we get a lower bound on $ t' $. Thus for any $ t $ which does not exceed that lower bound, $ \phi(\lambda t) $ is not a critical value of $ f_t $ and so all branches of $ f^{-1}_t $ are well defined and holomorphic in a neighborhood of the graph of $ \phi $. Therefore, choose $ g_t $ to be the branch of $ f^{-1}_t $ for which $ g_0(f_0(z))=z $, then the function $ \psi(t)= g_t(\phi(\lambda t)) $ is well defined from $ t=0 $ up to $ |t|<\eta$  satisfying $ \psi(0)=z_0 $ and the graph of $ \psi $  containing in the Fatou set, where $ \eta$ is the lower bound from (4.5) and (4.6).  We know that $ \psi $ is also bounded by $ R $, since otherwise $ \phi $ would not be bounded by $ R $. To avoid the case $ |t'|\geq \delta_1 $, we can shrink $ \beta $ such that $ \beta \;d(z_0,J(p))^{d_1(d_1+1)}< \delta_1 $ for all $ z_0 $. Thus $ |t'|\geq\delta_1 $ implies $ |t'|\geq \beta \;d(z_0,J(p))^{d_1(d_1+1)} $. Thus at least one of (4.5) and (4.6) holds.
		
	\end{proof}
	\textit{Proof of Theorem 4.3.} In the following we fix an invariant Fatou component $ U $ of $ p $, denote the basin of $ U $ by $ B $ (If $  B $ is the basin of infinity we let $ B $ be contained in $ \left\lbrace |z|<R\right\rbrace $ ). We can shrink $ \epsilon $ to ensure that the set $ \left\lbrace z\in B,d(z,J(p))< 2\epsilon\right\rbrace $ is contained in $ \left\lbrace |z|<R\right\rbrace $. In either case we first construct a subset $ W $ of $ B $, satisfies the following conditions,

	\smallskip
	\par (1) $ W $ eventually traps the forward orbit of any point in $ B $.
	
	\smallskip
	\par (2) $ W $ contains the compact  subset $ \left\lbrace z\in B,d(z,J(p))\geq \epsilon\right\rbrace  $.
	
	\smallskip
	\par (3) Theorem 4.3 holds for $ z\in W $.
	
	\medskip
	\par Finally we use the Pull  Back Lemma to prove Theorem 4.3 holds for $ z\in B$.
	
	\medskip	
	\noindent
	{\bf  Step 1:} Construction of $ W $. We split the argument in several cases.
	
	\medskip
	\noindent
	{\bf $\bullet$ $ U $ is an immediate attracting basin.} Let $ \omega \subset U$ be a compact neighborhood of the attracting fixed point. We set $ W= \left\lbrace z\in B,d(z,J(p))\geq \epsilon\right\rbrace  \cup \omega $, then $ W $ automatically satisfies (1) and (2). Since $ W $ is also compact and contained in $ F(P) $, there is a lower bound $ a>0 $ such that $ r(z)\geq a $ for every $ z\in W $. So there exist $ k>0 $ such that $ r(z)\geq k\;d(z,J(p)) $ for $z\in W$.
	\medskip
	
	\noindent
	{\bf $\bullet$ $ U $ is the attracting basin of $ \infty $.} We set $ W= \left\lbrace z\in B,d(z,J(p))\geq \epsilon\right\rbrace $, then $ W $ automatically satisfies (1) and (2). There is a lower bound $ a>0 $ such that $ r(z)\geq a $ for every $ z\in W $. So there exist $ k>0 $ such that $ r(z)\geq k\;d(z,J(p)) $.
	\medskip
	
	\noindent
	{\bf $\bullet$ $ U $ is an immediate parabolic basin.} Let $ Q $ be the associated attracting petal of Theorem 3.3. We set $ W= \left\lbrace z\in B,d(z,J(p))\geq \epsilon\right\rbrace  \cup Q $, then $ W $ automatically satisfies (1) and (2). By Theorem 3.3 there is a lower bound $ a>0 $ such that $ r(z)\geq a $ for every $ z\in P $. Thus  there is a lower bound $ b>0 $ such that $ r(z)\geq b $ for every $ z\in W $. So there exist $ k>0 $ such that $ r(z)\geq k\;d(z,J(p)) $.
	\medskip
	
	\noindent
	{\bf $\bullet$ $ U $ is a Siegel disk.} We set $ W=U\cup  \left\lbrace z\in B,d(z,J(p))\geq \epsilon\right\rbrace $, then $ W $ automatically satisfies (1) and (2). To prove (3), it is enough to prove (3) for $ z\in U $. 
	\begin{lemma}
		Let $ U $ be a Siegel disk, then there are constants $ k>0 $, $ l>0 $ such that for any point $ z\in U  $, 
		\begin{equation*}
		r(z)\geq k\;d(z,J(p))^l.
		\end{equation*}
		Further more $l$ only depends on $p$.
	\end{lemma}
\begin{proof}
Since the technique of the proof is similar to that of Theorem 4.3, we postpone the proof to the end of this subsection.
\end{proof}
	\medskip	
\noindent
{\bf  Step 2:} Pull back argument.

\medskip
	\par We already have the estimate for $ z\in W $. For every $ z_0\in B \backslash W $, let $ \left\lbrace z_i\right\rbrace_{i\geq 0} $ be its forward orbit, and let $ n $ be the smallest integer such that $ z_n $ lies in $ W$. Let $ m $ be the smallest integer such that case (4.5) does not happen, if this $ m $ dose not exist, let $ m=n $, in either case we have
	\begin{equation*}
	r(z_m)\geq k \;d(z_m,J(p))^l,
	\end{equation*}
	for some $ k,l>0 $,
 and for all $ z_i,\; 0\leq i\leq m-1 $, we have
	\begin{equation}
	r(z_i)\geq \frac{\alpha}{|\lambda|} r(z_{i+1}))d(z_i,C(p))^{d_1(d_1+1)}. 
	\end{equation} 
	\smallskip

	By (4.8) we have 
	\begin{align*}
	\log r(z_i)&\geq  \log r(z_{i+1})+\log d(z_i,C(p))^{d_1(d_1+1)}+\log\frac{\alpha}{|\lambda|}\\&= \log r(z_{i+1})-d_1(d_1+1) k(z_i)+\log\frac{\alpha}{|\lambda|},
	\end{align*}
	for all $ 0\leq i\leq m-1 $, where $ k(z_i) $ is as in Lemma 2.5.
	\par Thus we have
	\begin{align*}
	\log r(z_0)&\geq  \log r(z_{m})-d_1(d_1+1)\sum_{i=0}^{m-1} k(z_i)+m\log\frac{\alpha}{|\lambda|}.
	\end{align*}
	By Lemma 2.5 there exist a subset $ \left\lbrace i_1, \cdots, i_{q'}\right\rbrace \subset \left\lbrace 0, 1, \cdots, m-1 \right\rbrace $ such that
	\begin{equation*}
	\sum_{i=0}^{m-1}k(z_i)-\sum_{\alpha=1}^{q'} k(z_{i_\alpha})\leq Qm.
	\end{equation*}
	Therefore we have
	\begin{align}
	\log r(z_0)&\geq\log r(z_{m})-d_1(d_1+1) \sum_{\alpha=1}^{q'} k(z_{i_\alpha})-d_1(d_1+1)Qm+m\log\frac{\alpha}{|\lambda|}\notag\\&\geq\log r(z_{m})+d_1(d_1+1) \sum_{\alpha=1}^{q'}\log d(z_{i_\alpha},J(p))-d_1(d_1+1)Qm+m\log\frac{\alpha}{|\lambda|}.
	\end{align}
	By Corollary 2.7 we have for each $ i_\alpha $,
	\begin{align*}
	\log d(z_{i_\alpha},J(p))&\geq \frac{1}{\rho}\log d(z_0,J(p))-\frac{1}{\rho}i_\alpha\log s\\&\geq \frac{1}{\rho}\log d(z_0,J(p))-\frac{1}{\rho}m\log s.
	\end{align*}
   Likewise we have,
    \begin{align*}
    \log r(z_{m})&\geq \log k+l\log d(z_m,J(p))\\&\geq \log k+\frac{l}{\rho}\log d(z_0,J(p))-\frac{l}{\rho}m\log s.
    \end{align*}
		Thus applying the estimates of $ \log d(z_{i_\alpha},J(p)) $ and $ \log r(z_{m}) $ to (4.9) gives 
	\begin{align*}
	\log r(z_0)&\geq\log r(z_{m})+d_1(d_1+1) \sum_{\alpha=1}^{q'}\log d(z_{i_\alpha},J(p))-d_1(d_1+1)Qm+m\log\frac{\alpha}{|\lambda|}\\&\geq\log k+\frac{l+qd_1(d_1+1)}{\rho}\log d(z_0,J(p))-\frac{l+qd_1(d_1+1)}{\rho}m\log s-d_1(d_1+1)Qm+m\log\frac{\alpha}{|\lambda|}.
	\end{align*}
	Let us now fix $ \lambda_1 $ so small such that
	\begin{equation}
	\log\frac{\alpha}{\lambda_1}\geq \frac{l+qd_1(d_1+1)}{\rho}\log s+d_1(d_1+1)Q,
	\end{equation}
	then for every $ |\lambda|<\lambda_1 $ we have
	\begin{equation*}
	\log r(z_0)\geq \log k+\frac{l+qd_1(d_1+1)}{\rho}\log d(z_0,J(p)),
	\end{equation*}
which is equivalent to
	\begin{equation*}
	r(z_0)\geq k\; d(z_0,J(p))^{l'},
	\end{equation*}
	where $ l'=\frac{l+qd_1(d_1+1)}{\rho} $.
	
	\par We have shown that there are constants $ k>0 $, $ l'>0 $ such that $ r(z)\geq k\;d(z,J(p))^{l'}$  for $ z\in B $, and $l'$ only depends on $p$,  this finishes the proof of Theorem 4.3.
	\qed
	\medskip
	\par 	\textit{Proof of Lemma 4.5} 
	It is enough to prove the estimate for an invariant subset $U_\epsilon\subset  U \backslash  \left\lbrace z\in B,d(z,J(p))\geq \epsilon\right\rbrace  $. First note that the conclusion of Lemma 4.4 holds for all $ z_0\in U_\epsilon $, since for all $ z_0\in U_\epsilon $ the condition $ p(z_0)\in V $ holds. Since $ U $ is a Siegel disk, the forward orbit $ \left\lbrace z_n\right\rbrace_{n\geq 0}$  lies in a compact subset $ S $ of $ U $, where  $ z_n=p^n(z_0) $. Thus there is a lower bound $ a >0$ such that $ r(z)\geq a $ for $ z\in S $, $ a $ depending on $ S $. By Lemma 4.4 there are two cases,
	\medskip
	
	\par (1) There is no such integer $n$ that $ r(z_n)\geq \beta \;d(z_n,J(p))^{(d_1+1)d_1}  $, thus all $ z_n $ satisfy $ 	r(z_n)\geq \frac{\alpha}{|\lambda|} r(z_{n+1})d(z_n,C(p))^{d_1(d_1+1)}.$
	\smallskip
	\par (2) There is an integer $ n $ such that $ r(z_n)\geq \beta \;d(z_n,J(p))^{d_1(d_1+1)}  $.
	\medskip
	
	In  case (1) for every $ i\geq 0 $
	\begin{align*}
	\log r(z_i)&\geq  \log r(z_{i+1})+\log d(z_i,C(p))^{d_1(d_1+1)}+\log\frac{\alpha}{|\lambda|}\\&= \log r(z_{i+1})-d_1(d_1+1) k(z_i)+\log\frac{\alpha}{|\lambda|}.
	\end{align*}
	Thus we have for every $ n\geq 0 $,
	\begin{align*}
	\log r(z_0)&\geq  \log r(z_{n})-d_1(d_1+1)\sum_{i=0}^{n-1} k(z_i)+n\log\frac{\alpha}{|\lambda|}\\&\geq \log a+\frac{qd_1(d_1+1)}{\rho}\log d(z_0,J(p))-\frac{qd_1(d_1+1)}{\rho}n\log s-d_1(d_1+1)Qn+n\log\frac{\alpha}{|\lambda|}.
	\end{align*}
	Let us now fix $ |\lambda_1| $ so small such that
	\begin{equation}
	\log\frac{\alpha}{\lambda_1}\geq \frac{l+qd_1(d_1+1)}{\rho}\log s+d_1(d_1+1)Q+1,
	\end{equation}
	thus for every $ |\lambda|<\lambda_1 $ we have
	\begin{align*}
	\log r(z_0)&\geq\log a+\frac{qd_1(d_1+1)}{\rho}\log d(z_0,J(p))+n.
	\end{align*}
	Let $ n\to\infty $ then $r(z_0)$ can be arbitrary large, which is a contradiction, thus actually case (1) can not happen.
	
	\par For the case (2), the proof is same as the proof of Theorem 4.3, thus the proof is complete.
	\qed
	\begin{remark}
	The constant $ \lambda_1 $ appearing in Theorem 4.3 is invariant under  local coordinate change $ t\mapsto \phi(t) $ with $\phi(0)=0$ and $ \phi'(0)=1 $. To see this from Lemma 4.4 and Remark 4.2 we know that $ \alpha $ is invariant since it only depends on $ p $ and $ K $. By (4.10) and (4.11) $ \lambda_1 $ only depends on $ \alpha $ and $ p $, hence $ \lambda_1(f) $ is invariant.
	\end{remark}
	
	\section{Estimate of size of forward images of vertical Fatou disks}
	In this section we adapt the DPU Lemma to the attracting local polynomial skew product case, to show that the size of forward images of a wandering vertical Fatou disk shrinks slowly. We begin with two classical lemmas. We follow Lilov's presentation.
	
		\begin{lemma}
		There exist $ c_0>0 $ depending only on $ p $ and $ \delta_2>0 $ such that when $ |t_0|<\delta_2 $, let $ \Delta(x,r)\subset \left\lbrace t=t_0 \right\rbrace  $ be an arbitrary vertical disk, then $ P(\Delta(x,r)) $ contains a disk $ \Delta(P(x),r')\subset \left\lbrace t=\lambda t_0 \right\rbrace  $ of radius $ \geq c_0r^d $.
	\end{lemma}
	\begin{proof}
		For fixed $ x=(t,z)$ satisfying $|t|<\delta_2 , z\in \mathbb{C}$, and for fixed $r>0 $, define a function 
		\begin{equation*}
		f_{t,z,r}(w)=\frac{1}{rM_{t,z.r}}\left( f_t(z-rw)-f_t(z)\right)
		\end{equation*}
		which is a polynomial defined on the closed unit disk $ \overline D(0,1) $. The positive number $ M_{t,z,r} $ is defined by
		\begin{equation*}
		M_{t,z,r}=\sup_{w\in \pi_2(\overline\Delta(x,r))}|f'_t(w)|.
		\end{equation*}
		
		\par Let $ A $ be the  finite dimensional normed space containing all polynomials with degree $ \leq d $ on $\overline D(0,1) $, equipped with the uniform norm. Since $ |f'_{t,z,r}(w)|\leq 1 $ on $\overline D(0,1) $, the family $\left\lbrace f'_{t,z,r}\right\rbrace$ is bounded in $ A $. Notice that $ f_{t,z,r}(0)=0 $, so that $\left\lbrace f_{t,z,r}\right\rbrace$ is also bounded in $ A $. The closure of $\left\lbrace f_{t,z,r}\right\rbrace$ contains no constant map since the derivative of constant map vanishes. but $ \sup_{\overline D(0,1)}|f'_{t,z,r}(w)|=1 $.
		\par Now suppose that there is a sequence $ \left\lbrace f_{t_n,z_n,r_n} \right\rbrace $ such that $ f_{t_n,z_n,r_n}(D(0,1)) $ does not contains $ D(0,\delta_n) $, with $ \delta_n\to 0 $. We can take a sub-sequence $ f_{t_n,z_n,r_n}\rightarrow g $ , where $ g $ is a non-constant polynomial map with $ g(0)=0 $. Therefore by open mapping Theorem $ g(D(0,\frac{1}{2})) $ contains $ D(0,\delta) $ for some $ \delta>0 $. Then for $ n $ large enough $ f_{t_n,z_n,r_n}(D(0,1)) $ also contains $ D(0,\delta) $, which is a contradiction.  Therefore for all parameter $ \left\lbrace t,z,r\right\rbrace $, $ f_{t,z,r}(D(0,1)) $ contain a ball $ D(0,\delta) $, which is equivalent to say that
		\begin{equation}
		\Delta(P(x), \delta r M_{t,z,r})\subset P(\Delta(x,r)).
		\end{equation}
		\par Next we estimate $ M_{t,z,r} $ from below. Let $ z_1(t), z_2(t),..., z_{d-1}(t) $ be all zeroes of $ f'_t(z) $. Then $ f'_t(z)=da_d(t) (z-z_1(t))\cdots (z-z_{d-1}(t)) $. We choose $ \delta_2 $ small such that $ c_0=\inf_{|t|\leq \delta_2} |da_d(t)|>0 $.  Then we have 
		\begin{align*}
		M_{t,z,r}=\sup_{w\in \pi_2(\overline\Delta(x,r))}|f'_t(w)|&=\sup_{w\in \pi_2(\overline\Delta(x,r))}|da_d(t) (z-z_1(t))\cdots (z-z_{d-1}(t))|\\&\geq  c_0r^{d-1} \sin^{d-1}\frac{\pi_2}{d-1},
		\end{align*}
		this with (5.1) finishes the proof.
	\end{proof}	

	\begin{lemma}
			There exist $ 0<c<c_0, \;\delta_2>0 $ such that if a vertical disk $ \Delta(x,r)\subset \left\lbrace t=t_0 \right\rbrace  $ satisfies $ \Delta(x,r)\subset\left\lbrace |z|<R\right\rbrace $, $ |t_0|<\delta_2 $ and $ \eta=d(\Delta(x,r),\left\lbrace t=t_0\right\rbrace\cap \text{Crit}(P))>0$, then $ P(\Delta(z,r)) $ contains a disk $ \Delta(P(x),r')\subset \left\lbrace t=\lambda t_0 \right\rbrace  $ of radius $ \geq c\eta^{2d-2}r $.
	
	\end{lemma}
	\begin{proof}
		Let $ V=\left\lbrace x_0=(t_0,z_0): \; |t_0|<\delta_2, |z_0|<R, d(x_0,\left\lbrace t=t_0\right\rbrace \cap \text{Crit}(P))>\eta\right\rbrace  $, and set
		\begin{equation*}
		M_1=\inf_V\left|\frac{\partial f}{\partial z} \right|>0,\;M_2=\sup_V \left|\frac{\partial^2 f}{\partial z^2} \right|<\infty,
		\end{equation*}
		here $ M_1 $ depends on $ \eta $ but $ M_2 $ does not.
		\par Thus for $ \Delta(x_0,r)\subset \left\lbrace t=t_0 \right\rbrace  $ satisfying $ \Delta(x_0,r)\subset\left\lbrace |z|<R\right\rbrace $ and $ \eta=d(\Delta(x_0,r),\left\lbrace t=t_0\right\rbrace\cap \text{Crit}(P))>0 $,
	 we have $\Delta=\Delta(x_0,r)\subset V\cap\left\lbrace t=t_0\right\rbrace$. Pick an arbitrary $ a $ in the interior of $ \pi_2(\Delta) $. Then for all $ z\in \partial\pi_2( \Delta) $, we let 
		\begin{equation*}
		h(z)=f_{t_0}(z)-f_{t_0}(a)=f'_{t_0}(z)(z-a)+\frac{1}{2}(z-a)^2g(z).
		\end{equation*}
		We know $ g(z) $ satisfies $ |g(z)|\leq M_2 $, so that
		\begin{align*}
		|f'_{t_0}(z)(z-a)|\geq M_1|z-a|\geq M_1\frac{|z-a|^2}{2r}.
		\end{align*}
		In the case $ r\leq \frac{M_1}{2M_2} $ we have
		\begin{equation*}
		|f'_{t_0}(z)(z-a)|\geq M_2|z-a|^2> \frac{1}{2} |(z-a)^2g(z)|.
		\end{equation*}
		Thus by Rouché's Theorem the function $ h(z) $ has the same number of zero points as $ f'_{t_0}(z)(z-a) $, thus $ h(z) $ has exactly one zero point $ \left\lbrace z=a \right\rbrace $ . Since $ a\in \pi_2(\Delta) $ is arbitrary we have $ f_{t_0} $ is injective on $ \Delta $. The classical Koebe's one-quarter Theorem shows that $ P(\Delta(x_0,r)) $ contains a disk with radius at least
		\begin{equation}
		\frac{1}{4}\left|\frac{\partial f_{t_0}}{\partial z} (z_0)\right|r.
		\end{equation}
		Now we estimate $ \left|\frac{\partial f_{t_0}}{\partial z} (z_0)\right| $ from below. Let $ z_1(t), z_2(t),..., z_{d-1}(t) $ be all zeroes of $ f'_t(z) $. Then $ f'_t(z)=da_d(t) (z-z_1(t))\cdots (z-z_{d-1}(t)) $. We choose $ \delta_2 $ such that $ c_0=\inf_{|t|\leq \delta_2} |da_d(t)|>0 $. We have for every $ 1\leq i\leq d-1 $, $ |z_0-a_i(t_0)|\geq \eta $. Thus we have
		\begin{equation*}
\left|\frac{\partial f_{t_0}}{\partial z} (z_0)\right|=|da_d(t_0) (z_0-z_1(t_0))\cdots (z_0-z_{d-1}(t_0))|\geq c_0\eta^{d-1},
		\end{equation*}
	this with (5.2) gives
	\begin{equation*}
	r'\geq \frac{1}{4}c_0\eta^{d-1}r.
	\end{equation*}
	In the case $ r\geq \frac{M_1}{2M_2} $, by the same argument we have 
		\begin{equation}
	r'\geq \frac{1}{4}c_0\eta^{d-1}\frac{M_1}{2M_2}\geq  \frac{1}{8M_2}c_0^2\eta^{2d-2}.
	\end{equation}
	Setting $ c=\frac{1}{2}\min\left\lbrace \frac{c_0}{4R^{d-1}}, \frac{1}{8RM_2}c_0^2\right\rbrace  $ we get the conclusion.
	\begin{remark}
We note that $ c$ is invariant under a local coordinate change of the form $ t\mapsto \phi(t) $ with $ \phi(0)=0 $ and $ \phi'(0)=1 $. To see this, we know $ c_0 $ and $ R $ are invariant under a local coordinate change of the form $ t\mapsto \phi(t) $ with $ \phi(0)=0 $ and $ \phi'(0)=1 $, and by shrinking $ \delta_2(\phi) $ we have
\begin{equation*}
\sup_V\left|\frac{\partial^2f(\phi(t),z)}{\partial z^2}\right|\leq 2M_2,
\end{equation*}thus from (5.3) and $ c=\frac{1}{2}\min\left\lbrace \frac{c_0}{4R^{d-1}}, \frac{1}{8RM_2}c_0^2\right\rbrace  $ we get that $ c $ is invariant.
	\end{remark}
		\medskip
	\end{proof}

	\par Now we show that the size of forward images of a wandering vertical Fatou disk shrinks slowly.  We begin with a definition.
	\begin{definition}
	Define the \textbf{inradius}  $ \rho $ as follows: for a domain $ U \subset \mathbb{C} $, for every $ z\in U \subset \mathbb{C}$, define
	\begin{equation*}
	\rho(z,U)=\sup\left\lbrace r>0 |\;D(z,r)\subset U\right\rbrace ,
	\end{equation*}
	here $ D(z,r) $ is a disk centered at $ z $ with radius $ r $.
	\end{definition}
	
	\begin{proposition}
		Let $ \Delta_0 \subset \left\lbrace t=t_0\right\rbrace  $ be a wandering vertical Fatou disk centered at $ x_0=(t_0,z_0) $. Let $x_n= (t_n,z_n)=P^n(x_0) $. 
		Set $ \Delta_n=P^n(\Delta_0) $  for every $ n \geq 1$ and let $  \rho_n=\rho(z_n,\pi_2(\Delta_n)) $. There is a constant $ \lambda_2(f)$ such that for fixed $ |\lambda|<\lambda_2 $, we have
		\begin{equation*}
		\lim_{n\to \infty}\frac{|\lambda|^n}{\rho_n}=0.
		\end{equation*}
		
	\end{proposition}
	\begin{proof}
	 Let $ \lambda_3$ be a positive constant to be determined. It is sufficient to prove the result by replacing $ \Delta_n $ by $ \Delta_n\cap\Delta(x_n,\lambda_3^{n+1}) $. In the following we let $ \Delta_n $ always be contained in  $ \Delta(x_n,\lambda_3^{n+1}) $.
		
		\par Without loss generality we can assume that $ |t_0|<\min\left\lbrace \delta_1,\delta_2, \lambda_3\right\rbrace $, where $ \delta_1 $ is the constant in Lemma 4.1 and  $ \delta_2 $ is the constant in Lemma 5.1 and Lemma 5.2 . Let $ N $ be a fixed integer such that $ N>d^q+1 $, where $ q $ is the number of critical points lying in $ J(p) $. Let $ K=\left\lbrace |t|<\min\left\lbrace \delta_1,\delta_2, \lambda_3\right\rbrace  \right\rbrace \times \left\lbrace |z|<R \right\rbrace $ be a relatively compact subset of $ \mathbb{C}^2 $ such that for $ (t,z)\notin K $, $|f(t,z)|\geq 2|z| $. Since the orbits of points in $ \Delta_0 $ cluster only on $ J(p) $,we have $ \Delta_n\in K $  for every $ n $ . We need the following lemma:
		\begin{lemma}
		 There is a constant $ M>0 $ such that if $ |\lambda|<\lambda_3 $, for every $ n $ and for every $ x'=(t_n,w_n)\in \Delta_n $, for every integer $ m $, letting $ (t_{n+m},w_{n+m})=P^m(x') $ we have,
		\begin{equation}
		|w_{n+m}-p^{m}(z_n)|\leq M^m \lambda_3^{n+1}.
		\end{equation}
	\end{lemma}
\begin{proof}
		We prove it by induction. Let $ M $ satisfying for $ (t,z)\in K $, $ \left| \frac{\partial f(t,z)}{\partial t}\right| \leq \frac{M}{2} $ and $ \left| \frac{\partial f(t,z)}{\partial z}\right| \leq \frac{M}{2}$. We can also assume $ M $ is larger than the constant $ K $ in Lemma 4.1. Thus   For $ m=0 $ it is obviously true. Assume that when $ m=k-1 $ is true, we have
		\begin{align*}
		|w_{n+k}-p^k(z_n)|&=|f(t_{n+k-1},w_{n+k-1})-f(0,p^{k-1}(z_n))|\\&
		\leq \frac{M}{2}|t_{n+k-1}|+\frac{M}{2}|w_{n+k-1}-p^{k-1}(z_n)|\\& \leq \frac{M}{2} |\lambda|^{n+1}+\frac{M^k}{2} \lambda_3^{n+1}\\&\leq M^k \lambda_3^{n+1}.
		\end{align*}
		Thus for every $ m $, (5.4) holds.
		\end{proof}
	 Remark that when $ w_n=z_n $, the same argument gives
	\begin{equation*}
	|z_{n+m}-p^{m}(z_n)|\leq M^m |\lambda|^{n+1}.
	\end{equation*}
		\par Let  $ C(P) $ be the union of components of $ \text{Crit}(P) $ such that meet $ C(p)=\text{Crit}(p)\cap J(p) $ in the invariant fiber. For every point $ x\in \Delta_n $, we define $ k(x)=-\log d(x,C(P)\cap \left\lbrace t=t_n\right\rbrace ) $, and $ k_n=\sup_{x\in \Delta_n} k(x) $. (This definition allows $ k_n=+\infty $.) Recall that $ N $ is a fixed integer such that $ N>d^q+1 $. We are going to prove a two dimensional DPU Lemma for attracting polynomial skew products: 
		\begin{lemma}[\bf Two Dimensional DPU Lemma] \label{Two Dimensional DPU Lemma}
		Let $ |\lambda|<\lambda_3 $, then for every $ N^k\leq n<N^{k+1} $, there is a subset \[  \left\lbrace \alpha_1,\cdots,\alpha_{q'}\right\rbrace\subset \left\lbrace N^k-1, N^k,\cdots, n-1\right\rbrace  \] and a  constant $ Q >0$ such that
		\begin{equation}
		\sum_{i=N^k-1}^{n-1} k_i-\sum_{i=1}^{q'} k_{\alpha_i}\leq Q(n-N^k+1),
		\end{equation}
		here  $ k $ is an arbitrary integer, $ q' \leq q$ is an integer. Recall that $ q $ is the number of critical points lying in $ J(p) $. 
		\end{lemma}
	\begin{proof}
		 Recall that the DPU Lemma implies that there is a subset \[  \left\lbrace \alpha_1,\cdots,\alpha_{q'}\right\rbrace\subset \left\lbrace N^k-1, N^k,\cdots, n-1\right\rbrace  \] and a  constant $ Q >0$ such that
		\begin{equation}
		\sum_{i=N^k-1}^{n-1} k(p^{i-N^k+1}(z_{N^k-1})) -\sum_{j=1}^{q'} k(p^{\alpha_j-N^k+1}(z_{N^k-1}))\leq \frac{Q}{2}(n-N^k+1).
		\end{equation}
	
		\par So it is sufficient to prove $ k_i\leq 2k(p^{i-N^k+1}(z_{N^k-1}))$ for every $ i $ not appearing in $ \left\lbrace \alpha_1,\ldots,\alpha_{q'}\right\rbrace$. This is equivalent to 
		\begin{equation}
		d(\Delta_i,C(P)\cap \left\lbrace t=t_i\right\rbrace)\geq  d(p^{i-N^k+1}(z_{N^k-1}),C(p))^2.	\end{equation}
		To prove (5.7), assume that $ d(p^{i-N^k+1}(z_{N^k-1}),C(p))=d(p^{i-N^k+1}(z_{N^k-1}),c_k) $ for some point $ c_k\in C(p) $, let $ C_k $ be the component of $ C(P) $ which meats $ c_k $ at invariant fiber, by (5.4) and Lemma 4.1  we have
		\begin{align*}
		d(\Delta_i,C(P)\cap \left\lbrace t=t_i\right\rbrace)&\geq d(p^{i-N^k+1}(z_{N^k-1}),C(p))-\sup_{x'\in \Delta_i}|\pi_2(x')-p^{i-N^k+1}(z_{N^k-1})|-|w_i-c_k| \\&\geq d(p^{i-N^k+1}(z_{N^k-1}),C(p))-M^{i-N^k+1} \lambda_3^{N^k}-M |\lambda|^{\frac{N^k}{d_1}},
		\end{align*}
		where $ w_i $ is $ \pi_2(C_k\cap\left\lbrace t=t_i\right\rbrace ) $.
		\par By $|\lambda|<\lambda_3 $ we have
		\begin{equation*}
		M^{i-N^k+1} \lambda_3^{N^k}+M |\lambda|^{\frac{N^k}{d_1}}= \left( M^{i-N^k+1}+M\right) \lambda_3^{\frac{N^k}{d_1}}.
		\end{equation*}
		Thus we have
			\begin{equation*}
		d(\Delta_i,C(P)\cap \left\lbrace t=t_i\right\rbrace)\geq d(p^{i-N^k+1}(z_{N^k-1}),C(p))-\left( M^{i-N^k+1}+M\right) \lambda_3^{\frac{N^k}{d_1}}.
		\end{equation*}
		To prove (5.7) it is sufficient to prove
		\begin{equation}
		\left( M^{i-N^k+1}+M\right) \lambda_3^{\frac{N^k}{d_1}}\leq d(p^{i-N^k+1}(z_{N^k-1}),C(p))-d(p^{i-N^k+1}(z_{N^k-1}),C(p))^2.
		\end{equation}
		By (5.6) we have
		\begin{equation*}
		d(p^{i-N^k+1}z_{N^k-1},C(p))\geq e^{-\frac{Q}{2}(n-N^k+1)},
		\end{equation*}
		thus it is sufficient to prove
		\begin{equation}
		(M^{n-N^k+1}+M)\lambda_3^{\frac{N^k}{d_1}}\leq e^{-\frac{Q}{2}(n-N^k+1)}-e^{-Q(n-N^k+1)}.
		\end{equation}
		We can always choose $ \lambda_3 $ sufficiently small to make (5.9) holds for all  $k\geq 0$. This ends the proof of the two dimensional DPU Lemma (5.5).
		\end{proof}
		\par By Lemma 5.1 and Lemma 5.2 there is a constant $ c>0 $ such that
		\begin{equation}
		\rho_{n+1}\geq ce^{-(2d-2)k_n}\rho_n.
		\end{equation}
		and 
		\begin{equation}
		\rho_{n+1}\geq c\rho_n^d.
		\end{equation}
		From the above we can now give some estimates of $ \rho_{n} $. Recall that $\rho_n$ is assumed smaller than $ |\lambda_3|^{n+1} $ otherwise we replace it by $ \min\left\lbrace \rho_n,\lambda_3^{n+1}\right\rbrace  $.
		\begin{lemma}
		There is a constant $ c_1>0 $ such that for $ N^k\leq n<N^{k+1} $, we have
		\begin{equation*}
		\rho_{n}\geq c_1^{N^k}\rho_{N^k-1}^{d^q}.
		\end{equation*}
		\end{lemma}
		\begin{proof}
 For $ N^k\leq i\leq n $,  if $ i-1\in \left\lbrace \alpha_1,\cdots, \alpha_q\right\rbrace $ we apply inequality (5.10), if $ i\notin \left\lbrace \alpha_1,\cdots, \alpha_{q'}\right\rbrace $ we apply inequality (5.11). Thus we have
		\begin{align*}
		\rho_n&\geq c^{n-\alpha_{q'}+1}\exp\left( -(2d-2)\sum_{j=\alpha_{q'}+1}^{n}k_j\right) \left( \cdots \left( c^{\alpha_1-N^k+1}\exp\left( -(2d-2)\sum_{j=N^k-1}^{\alpha_1-1}k_j\right) \rho_{N^k-1}\right) ^d\cdots\right) ^d\\&\geq c^{(n-N^k+1)d^q}\exp\left( -d^q \sum_{j=N^k-1}^{n-1} k_i+d^q\sum_{j=1}^{q'} k_{\alpha_j} \right) \rho_{N^k-1}^{d^q}\quad (\text{because}\; q'\leq q)\\&\geq c^{(n-N^k+1)d^q}\exp\left( -Qd^q\left( n-N^k+1\right) \right) \rho_{N^k-1}^{d^q}\qquad (\text{by Lemma 5.7})\\&\geq c^{N^{k+1}d^q}\exp\left( -Qd^qN^{k+1}\right) \rho_{N^k-1}^{d^q}.
		\end{align*}
		Setting $ c_1=\min\left\lbrace c^{Nd^q}e^{-QNd^q}, \lambda_3 \right\rbrace $ we get the desired conclusion.
		\end{proof}
	\begin{lemma}
		For $ N^k\leq n<N^{k+1} $, $ \rho_0\leq \lambda_3 $ we have
		\begin{equation*}
	\rho_n\geq	c_1^{N^{k+1}}\rho_0^{d^{q(k+1)}}.
		\end{equation*}
		\begin{proof}
		By iterating Lemma 5.8 we get that
		\begin{equation*}
		\rho_{N^k-1}\geq c_1^{\frac{N^k-d^{qk}}{N-d^q}}\rho_0^{d^{qk}}\geq c_1^{N^k}\rho_0^{d^{qk}},
		\end{equation*}
			so that  
		\begin{equation*}
		\rho_n\geq c_1^{N^k}\rho_{N^k-1}^{d^q}\geq c_1^{N^{k+1}}\rho_0^{d^{q(k+1)}},
		\end{equation*}
		this finishes the proof.
		\end{proof}
	\end{lemma}

		Now we can conclude the proof of Proposition 5.5. For $ N^k\leq n<N^{k+1} $ we get
		\begin{equation*}
		\frac{|\lambda|^n}{\rho_n}\leq \frac{|\lambda|^{N^k}}{c_1^{N^{k+1}}\rho_0^{d^{q(k+1)}}}.
		\end{equation*}
	Choosing $ \lambda_2 $ small such that 
	\begin{equation}
	\lambda_2<c_1^N,
	\end{equation}  since $ N>d^q+1 $ we deduce that for every $ |\lambda|<\lambda_0 $,
		\begin{equation*}
		\lim_{k\to \infty}\frac{|\lambda|^{N^k}}{c_1^{N^{k+1}}\rho_0^{d^{q(k+1)}}}=0,
		\end{equation*}
	finally $\lim_{n\to \infty}\frac{|\lambda|^n}{\rho_n}=0$, which finishes the proof.
		
	\end{proof}
\begin{remark}
	The constant $ \lambda_2 $ appearing in Proposition 5.5 is invariant under a local coordinate change of the form  $ t\mapsto \phi(t) $ with $\phi(0)=0$ and $ \phi'(0)=1 $. To see this we know that by (5.9) $ \lambda_3 $ depends only on $ M $ and $ p $, $M$ can be dealt with by replacing it everywhere by $2M$ (see Remark 4.2), so that $ \lambda_3 $ is invariant. By putting $ c_1=\min\left\lbrace c^{Nd^q}e^{-QNd^q}, \lambda_3 \right\rbrace $ we get that $ c_1 $ is invariant. Then by (5.12) we get that $ \lambda_2 $ is invariant.
\end{remark}	
	\medskip
	\begin{corollary}
	In the same setting as Proposition 5.5, for every $ l>0 $, if $\lambda$ is chosen sufficiently small, we have
	\begin{equation*}
	\lim_{n\to \infty}\frac{|\lambda|^n}{\rho_n^l}=0.
	\end{equation*}
	\end{corollary}	
	\begin{proof}
		By Proposition 5.5 if $ |\lambda|<\lambda_2 $, then $\lim_{n\to \infty}\frac{|\lambda|^n}{\rho_n}=0$ holds.  For any $ l>0 $, we then let $ |\lambda |$ smaller than $\lambda_2^l$ to make the conclusion holds. 	
\end{proof}

	\section{Proof of the non-wandering domain theorem}
	In this section we prove the non-existence of wandering Fatou components. Let us recall the statement
	\begin{theorem}[\bf No wandering Fatou components] \label{th_ No wandering Fatou components}
Let $P$ be a local polynomial skew product  with an attracting invariant fiber,
	\begin{equation*}
	P(t,z)=(\lambda t,f(t,z)).
	\end{equation*}
	Then for any fixed $ f $, there is a constant $ \lambda_0(f) >0$ such that if $ \lambda $ satisfies $ 0<|\lambda|<\lambda_0 $, every forward orbit of a vertical Fatou disk intersects a bulging Fatou component. In particular every Fatou component iterates to a bulging Fatou component, and there are no wandering Fatou components.
	\end{theorem}
	\begin{proof}
		We argue by contradiction. Suppose $ \Delta_0\subset \left\lbrace t=t_0\right\rbrace  $ is a  vertical disk lying in a Fatou component which does not iterate to a bulging Fatou component. Without loss generality we may assume $ |t_0|<\min\left\lbrace 1,\delta_1,\delta_2, \lambda_3\right\rbrace $. By Remark 2.2, $ \Delta_0 $ is a vertical Fatou disk. Let $ x_0=(t_0,z_0)\in \Delta_0 $ be the center of $ \Delta_0 $ and set $ x_n=(t_n,z_n)=P^n(t_0,z_0)$ and $\Delta_n=P^n(\Delta_0) $. We divide the proof into several steps, We set $ \rho_n=\rho(z_n,\pi_2(\Delta_n)) $ as before and assume that $ \rho_0\leq \lambda_3 $. Notice that $ \Delta_0 $ can not be contained in the basin of infinity, thus $  \Delta_n $ is uniformly bounded. Let $ \lambda_0<\min\left\lbrace \lambda_1, \lambda_2^l \right\rbrace  $, where $ \lambda_1$ and $\lambda_2$ come from Theorem 4.3 and Proposition 5.5. In the course of the proof we will have to shrink $ \lambda_0 $ one more time.
		\medskip
		
		\noindent
		{\bf $\bullet$ Step 1.} By Remark 3.6,  the orbits of points in $ \Delta_0 $ cluster only on $ J(p) $. 
		
		\medskip
		\noindent
		{\bf $\bullet$ Step 2.} We show that there exist $ N_0>0 $ such that when $ n\geq N_0 $, the projection $ \pi_2\left( \Delta(x_{n}, \frac{\rho_{n}}{4})\right)  $ intersects $ J(p) $. We determine $ N_0 $ in the following. Suppose  $ \pi_2\left( \Delta(x_{n}, \frac{\rho_{n}}{4})\right)  $ does not intersect $ J(p) $. Thus $ z_n \in F(p) $ and Theorem 4.3 implies $ r(z_{n})\geq k \;d(z_{n},J(p)) ^l$, then we have
		\begin{align*}
		\frac{|t_{n}|}{r(z_{n})}\leq \frac{|t_{n}|}{k \;d(z_{n},J(p)) ^l}\leq \frac{4^l|t_{n}|}{k \;\rho_{n} ^l}.
		\end{align*}
		By  Corollary 5.11 we can let $ N_0 $ large enough so that for all $ n\geq N_0 $, $ \frac{4^l|t_{n}|}{k \;\rho_{n} ^l}<1 $. From the definition of $ r(z_n) $ we get a horizontal holomorphic disk defined by $ \phi(t) $, $ |t|<r(z_n) $  contained in the bulging Fatou components, with $ \phi(0)=z_{n} $, and $ t_{n} $ is in the domain of $ \phi $.  Then we have
		\begin{align*}
		|\phi(t_{n})-z_{n}|=|\phi(t_{n})-\phi(0)|\leq 2R\frac{|t_{n}|}{r(z_{n})}\leq 2R\frac{|t_{n}|}{k \;d(z_{n},J(p)) ^l}\leq
		2R\frac{4^l|t_{n}|}{k \;\rho_{n} ^l}.
		\end{align*}
		Again by Corollary 5.11,  we can let $ N_0 $ large enough  that for all $ n\geq N_0 $, $ 2R\frac{4^l|t_{n}|}{k \;\rho_{n} ^l}<\frac{\rho_{n}}{4} $. Thus $  \phi(t_{n})\in \Delta(x_{n},\frac{\rho_{n}}{4})\subset \Delta_{n} $. Since $ \phi(t_{n}) $ is contained in the bulging Fatou components that contains $ z_{n} $, this implies $ \Delta_{n} $ intersects the bulging Fatou component so it can not be wandering. This contradiction shows that $ \pi_2\left( \Delta(x_{n}, \frac{\rho_{n}}{4})\right)  $ intersects $ J(p) $. 
		\par Let $ y_{n}\in \Delta_{n} $ satisfies $\pi_2\left(y_{n} \right) \in \pi_2\left( \Delta(x_{n}, \frac{\rho_{n}}{4})\right) \cap J(p) $, then for all $ x\in \Delta(y_{n}, \frac{\rho_{n}}{4}) $ we have $ \rho(\pi_2(x),\pi_2(\Delta_{n}))\geq \frac{\rho_{n}}{2} $. 
		
		\medskip
		\noindent
		{\bf $\bullet$ Step 3.} We show that there is an integer $ N_1>N_0 $ such that for every $  x\in \Delta(y_{N_1}, \frac{\rho_{N_1}}{4}) $, for every $ m\geq 0 $, $ p^m(\pi_2(x))\in \pi_2(\Delta_{m+N_1} ) $, here $\pi_2\left(y_{N_1} \right) \in \pi_2\left( \Delta(x_{N_1}, \frac{\rho_{N_1}}{4})\right) \cap J(p) $. This means that the orbit of $ \pi_2(x) $ is always shadowed by the orbit of $ \Delta_{N_1} $, which will contradict the fact that $ \pi_2(\Delta_{m+N_1} ) $ intersects $ J(p) $.   To show this, we inductively prove the more precise statement that for fixed $ N>d^q+1 $, there exist a  large $ N_1=N^{k_0}-1>N_0 $, such that  for every $k\geq k_0$, $N^k\leq n<N^{k+1} $, we have
		\begin{equation}
		p^{n-N_1}(\pi_2(x))\in \pi_2(\Delta_{n} )
		\end{equation} and 
		\begin{equation}
		\rho'_n\geq c_2^{N^{k+1}}\rho_{0}^{d^{q(k+1)}},
		\end{equation}
		where  $ \rho'_n=\rho(p^{n-N_1}(\pi_2(x)),\pi_2(\Delta_{n}) $, $ c_2=\frac{c_1}{2} $ comes from Lemma 5.8 and Lemma 5.9. We will determine $k_0$ in the following. 
		\par From Lemma 5.9 we know that (6.1) and (6.2) hold for $ n=N_1 $. Assume that for some $k\geq k_0$, for all $ n\leq N^k-1 $, (6.1) and (6.2) holds. Then for $N^k\leq n<N^{k+1} $, let
	 $ y=P^{n-N^k+1}(t_{N^k-1}, p^{N^k-1-N_1}(\pi_2(x)) ) $, by Lemma 5.8 we have
		\begin{equation}
		\rho(y,\Delta_{n})\geq c_1^{N^k}\left( {\rho'}_{N^k-1}\right) ^{d^q}.
		\end{equation}
		To estimate the distance between $ \pi_2(y)$ and $p^{n-N_1}(\pi_2(x)) $, by Lemma 5.6 we have
		\begin{align}
		|\pi_2(y)-p^{n-N_1}(\pi_2(x))|\leq M^{n-N^k+1}|\lambda|^{N^k}.
		\end{align}
		From (6.3) and (6.4) we have
		\begin{align*}
		\rho'_n&\geq c_1^{N^k}\left( {\rho'}_{N^k-1}\right) ^{d^q}-M^{n-N^k+1}|\lambda|^{N^k}\\&\geq c_1^{N^k}(c_2^{N^{k}}\rho_{0}^{d^{qk}})^{d^q}-M^{n-N^k+1}|\lambda|^{N^k}\qquad\text{(By the induction hypothesis (6.2))}\\&
		\geq c_1^{N^k}c_2^{N^kd^q} \rho_{0}^{d^{q(k+1)}}-M^{n-N^k+1}|\lambda|^{N^k}.
		\end{align*}
		By the choice $ c_2=\frac{c_1}{2} $ we have
		
		\begin{equation*}
		\rho'_n\geq 2c_2^{N^{k+1}}\rho_{0}^{d^{q(k+1)}}-M^{n-N^k+1}|\lambda|^{N^k}.
		\end{equation*}
		To get (6.2) it is sufficient to prove
		\begin{equation*}
		c_2^{N^{k+1}}\rho_{0}^{d^{q(K+1)}}\geq M^{n-N^k+1}|\lambda|^{N^k}.
		\end{equation*}
		We take $ \lambda_0 $ sufficiently small such that
		\begin{equation}
		\lambda_0\leq (\frac{c_2}{M} )^{2N}.
		\end{equation}
		 Thus to prove (6.2) it is sufficient to prove that when $ |\lambda|<\lambda_0 $,
		\begin{equation}
\rho_{0}^{d^{q(k+1)}}\geq |\lambda|^{\frac{N^k}{2}}.
		\end{equation}
		Since  $ N>d^q+1$, we can choose $ k_0 $ large enough such that for every $ k>k_0 $ (6.6) holds. This finishes the induction.

		\par This shows that (6.1) and (6.2) are true for all $ n\geq N_1 $.
		
		\medskip
		\noindent
		{\bf $\bullet$ Step 4.} Since for every $  x\in \Delta(y_{N_1}, \frac{\rho_{N_1}}{4}) $, for every $ m\geq 0 $, $ p^m(\pi_2(x))\in \pi_2(\Delta_{m+N_1} ) $, and $  \Delta_n $ is uniformly bounded, the family $ \left\lbrace p^m\right\rbrace_{m\geq 0} $ restricts on $ D(\pi_2(y_{N_1}), \frac{\rho_{N_1}}{4})) $ is a normal family. Thus $ \pi_2(y_{N_1}) $ belongs to the Fatou set $ F(p) $, this contradicts to  $ \pi_2(y_{N_1})\in J(p) $. Thus the proof is complete.
	
\end{proof}

\begin{remark}
The constant $ \lambda_0 $ appearing in Theorem 6.1 is invariant under a local coordinate change of the form $ t\mapsto \phi(t) $ with $\phi(0)=0$ and $ \phi'(0)=1 $. To see this we know that the constants $ c_2=\frac{c_1}{2}$, $ M $ and $N$ are invariant under a local coordinate change of the form $ t\mapsto \phi(t) $ with $\phi(0)=0$ and $ \phi'(0)=1 $ ($M$ can be dealt with by replacing it everywhere by $2M$, see Remark 4.2). Then by (6.5) $ \lambda_0 $ only depends on $ c_2$, $M $ and $N$, thus $ \lambda_0 $ is invariant.
\end{remark}
	\begin{remark}
Lilov's Theorem can be seen as a consequence of Theorem 6.1. In fact, for the super-attracting case, the Fatou components of $ p $ bulge for a similar reason. Since when $ |t| $ is very small, the contraction to the invariant fiber is stronger than any geometric contraction $ t\mapsto \lambda t $, Theorem 4.3 and Proposition 5.5 follows easily. Thus following the argument of Theorem 6.1 gives the result.
	\end{remark}

\medskip
In the following theorem we show how the main theorem can be applied to globally defined polynomial skew products. 

\begin{theorem}
 Let 
\begin{equation*}
P(t,z)=(g(t),f(t,z)):\mathbb{C}^2\to \mathbb{C}^2
\end{equation*}  
be a globally defined polynomial skew product, where $ g,f $ are polynomials. Assume $\text{deg} \;f=d$ and the coefficient of the term $z^d$ of $f$ is non-vanishing, then there exist a constant $ \lambda_0(t_0,f)>0 $ depending only on $f$ and $ t_0 $ such that if $ g(t_0)=t_0 $ and $ |g'(t_0)|<\lambda_0 $ then there are no wandering Fatou components in $ B(t_0)\times \mathbb{C}, $ where $ B(t_0) $ is the attracting basin of $ g $ at $ t_0 $ in the $ t $-coordinate.
\end{theorem}
\begin{proof}
	First by a coordinate change $\phi_0:t\mapsto t+t_0 $, $ P $ is conjugated to 
	\begin{equation*}
	P_0:(t,z)\mapsto (g_0(t), f_0(t,z)),
	\end{equation*}
	where $ g_0(t)=g(t+t_0)-t_0 $, and $ f_0(t,z)=f(t+t_0,z) $. It is clear that $ \left\lbrace t=0 \right\rbrace  $ becomes an invariant fiber.
	\par By Koenig's Theorem we can introduce a local coordinate change $ \phi: t\mapsto \phi(t) $ with $\phi(0)=0$ and $ \phi'(0)=1 $ such that $ P_0 $ is locally conjugated to
	\begin{equation}
	(t,z)\mapsto (\lambda t, f_0(\phi(t),z)),
	\end{equation}
	where $ \lambda=g'(t_0) $.
	\par We have seen in Remark 6.2 that the constant $ \lambda_0(f) $ is invariant under a local coordinate change of the form $ t\mapsto \phi(t) $ with $\phi(0)=0$ and $ \phi'(0)=1 $. This means that for fixed $f$, for every such $ \phi $,
	\begin{equation*}
	P_\phi: (t,z)\mapsto (\lambda t, f(\phi(t),z))
	\end{equation*} has no wandering Fatou components when $ |\lambda|<
	\lambda_0(f) $. 
	Thus applying this to (6.7) when $|\lambda|= |g'(t_0)|<\lambda_0(f_0) $ we get the local skew product $(t,z)\mapsto (\lambda t, f_0(\phi(t),z))$ has no wandering Fatou components. Thus by conjugation $ P $ has no wandering Fatou components in a neighborhood of $ \left\lbrace t=t_0\right\rbrace $ , thus actually $P$ has no wandering Fatou components in $ B(t_0)\times \mathbb{C}, $ where $ B(t_0) $ is the attracting basin of $ g $ at $ t_0 $ in the $ t $-coordinate.
\end{proof}
\medskip
	\bibliographystyle{plain}
	
	\bibliography{Mybib}
	
\end{document}